\documentclass[11pt, a4paper]{amsart}

\usepackage[colorinlistoftodos]{todonotes}

\usepackage[utf8]{inputenc}
\usepackage{lmodern}
\usepackage{amssymb,amsfonts,amsthm,amsmath,bbm,mathrsfs,aliascnt,mathtools}
\usepackage[english]{babel}
\usepackage[colorlinks]{hyperref}

\usepackage{ulem}
\usepackage{nicefrac}

\usepackage{tikz} 
\usepackage{subcaption}

\usepackage{graphicx} %
\usepackage{sidecap, caption}
\usepackage{wrapfig}

\usepackage{multicol}
\usepackage{lipsum}
\usepackage{array}

\theoremstyle{plain}
\newtheorem{thm}{Theorem}[section]
\newtheorem{cor}[thm]{Corollary}
\newtheorem{lem}[thm]{Lemma}
\newtheorem{prop}[thm]{Proposition}

\theoremstyle{definition}

\theoremstyle{remark}
\newtheorem{remark}[thm]{Remark}

\DeclareMathOperator{\supp}{supp}

\newcommand{\eps}{\varepsilon}

\def\d{{\rm d}}

\newcommand{\esssup}{\text{ess\,sup}}

\begin{document}

\title[Unique ergodicity...]{Unique Ergodicity for Noninvertible Function Systems on an Interval}

\author[S.C. Hille]{Sander C. Hille}
\address{Mathematical Institute, Leiden University, P.O. Box 9512, 2300 RA Leiden, The Netherlands (SH)}
\email{shille@math.leidenuniv.nl}
\author[H. Oppelmayer]{Hanna Oppelmayer}
\address{Institut f\"ur Mathematik, Universit\"at Innsbruck, Technikerstrasse 13, A--6020 Insbruck, Austria (HO)}
\email{Hanna.Oppelmayer@uibk.ac.at}
\author[T. Szarek]{Tomasz Szarek}
\address{Institut of Mathematics Polish Academy of Sciences, Abrahama 18, Sopot, Poland (TS)}
\email{tszarek@impan.pl}

\thanks{T.S. was supported by the Polish NCN Grant 2022/45/B/ST1/00135.}

\date\today

\maketitle

\begin{abstract}
We study random dynamical systems of certain continuous functions on the unit interval. We use bounded variation to provide sufficient conditions for unique ergodicity of these systems. Several classes of examples are provided.
\end{abstract}

\section{Introduction}

Unique ergodicity of dynamical systems is a classical property, which has been intensively studied in various settings for many years. Already in the 1950s  Oxtoby \cite{O}, Kakutani \cite{Ka} and many others obtained deep results about uniquely ergodic dynamical systems.

Recently there has been rapid progress in studying ergodicity for random dynamical systems. It appeared that a combination of probabilistic and dynamical systems tools 
brings new insights into the complex behaviour of stochastic dynamical systems. 

Unique ergodicity of random homeomorphisms on intervals has been almost completely examined by Alseda and Misiurewicz \cite{AM}, Czudek and Szarek \cite{CS}, Deroin, Kleptsyn and Navas \cite{DKN},
Gharaei--Homburg \cite{GH}, Bara\'nski--\'Spiewak \cite{BS} and many others. In turn,
unique ergodicity for random homomorphisms on the circle was studied, e.g., by Malicet \cite{M}, Navas \cite{N}. 

It seems that  the case of non-invertible random maps has not been yet completely analysed. Some results  were obtained for certain classes of non-invertible iterated function systems (see \cite{BOS, HHOS, HKRVZ, Kl}).

In this note, we aim at formulating a new and quite general sufficient condition for unique ergodicity of non--invertible random maps. It makes use of variation of random composition.

Let $\mu$ be a Borel probability measure on the space $C([0, 1])$ of all continuous functions $g: [0, 1]\to [0, 1]$ equipped with the supremum norm $|\cdot|$. For brevity we will write $g h$ for the composition $g\circ h$ for
$g, h\in C([0, 1])$.

By $\Gamma$ we shall denote the topological support of $\mu$.
The pair $(\Gamma, \mu)$ will be called a {stochastic dynamical system on $C([0, 1])$}.

Let $\bf x$ be a $[0, 1]$--random variable, and let $\bf g$ be a $\Gamma$--random variable. Assume $\bf x$ and  $\bf g$ are independent and that they have distributions $\nu$ and $\mu$, respectively. Then we easily see that
the distribution of $\bf g(\bf x)$ is equal to $\mu\ast\nu$ (see for instance \cite{F}) where for every bounded Borel measurable function $f: [0, 1]\to\mathbb R$ we have
$$
\int_{[0, 1]} f(x) \d (\mu\ast\nu) (x)=\int_{\Gamma}\int_{[0, 1]} f(g(x))\d\nu(x)\d\mu(g).
$$

A Borel probabilistic measure $\nu$ is said to be $\mu$--{\it stationary} (or $\mu$--{\it invariant}) if 
$$
\mu\ast\nu=\nu
$$
or equivalently
$$
\nu(\cdot)=\int_{\Gamma} g\nu(\cdot)\,\d \mu(g),
$$
where $g\nu (A)=\nu(g^{-1}(A))$ for any Borel set $A\subset [0, 1]$.

A $\mu$--stationary measure is said to be $\mu$--{\it ergodic} if it cannot be written as a proper convex combination of two different $\mu$--stationary Borel probability measures.  Obviously, if $\nu$ is a unique $\mu$--stationary measure, then it is ergodic. Therefore, a unique $\mu$--stationary measure is called {\it uniquely ergodic}.
On the other hand,  it is well known that any two different $\mu$-ergodic measures are mutually singular  (see \cite{DPZ}).

The main purpose of the paper is to provide a condition on $(\Gamma,\mu)$ which implies unique ergodicity. An important ingredient is bounded variation of the random iteration as follows:
\begin{enumerate}
    \item[(BV)]\label{BV} \qquad\qquad $\mu^{\otimes \mathbb N}\Big(\{(g_1,g_2,\ldots )\in\Gamma^{\mathbb N}: \sup_{n\ge 1} \bigvee\limits_0^1 g_1g_2\dots g_n <\infty\}\Big)>0,$
\end{enumerate}
\vskip 3mm
where $\bigvee_0^1 g$ denotes the variation of the function $g$ on $[0,1]$ (see Section~\ref{prelim} for the definition).

Our main result is 
\begin{thm}\label{thm:main result}
	Let $(\Gamma, \mu)$ be a stochastic dynamical system on $C([0, 1])$ such that  property \hyperref[BV]{(BV)} is satisfied. Then the following statements hold:
    \begin{enumerate}
    \item[({\it i})] If $\nu_1,\nu_2$ are ergodic $\mu$-invariant Borel probability measures that are concentrated on $(0,1)$ and that satisfy
$$
		\{0,1\}\cap \supp(\nu_1)\cap\supp(\nu_2)\neq \emptyset,
$$
	then $\nu_1=\nu_2$.
    \item[({\it ii})] If $(\Gamma,\mu)$ has no atomic $\mu$-invariant probability measures with finite support in the open interval $(0,1)$,  then any two distinct $\mu$-invariant ergodic measures $\nu_1$ and $\nu_2$ on $(0,1)$ satisfy
    \[
        \supp(\nu_1)\cap\supp(\nu_2) =\emptyset.
    \] 
    \end{enumerate}
\end{thm}

Under  additional assumptions, we can conclude unique ergodicity:
\begin{cor}\label{clry:main application}
	Let $(\Gamma, \mu)$ be a stochastic dynamical system on $C([0, 1])$ that has property \hyperref[BV]{(BV)}. Assume that     
	\begin{enumerate}
		\item[$(\mathfrak{U}_1)$]  for every $x\in (0, 1)$ there exists $g\in\Gamma$ such that $g(x)<x$
		  
  or
		\item[$(\mathfrak{U}_2)$]  for every $x\in (0, 1)$ there exists $g\in\Gamma$ such that $g(x)>x$.
		\end{enumerate}
Then, on $(0,1)$, there is at most one $\mu$-invariant Borel probability measure.
\smallskip
If additionally $\nu(\{0\})=\nu(\{1\})=0$ for every $\mu$-invariant Borel probability measure, then $(\Gamma,\mu)$ is uniquely ergodic.
\end{cor}

\begin{cor}\label{cor1_16.09.68}
Let $(\Gamma, \mu)$ be a stochastic dynamical system on $C([0, 1])$ that has property \hyperref[BV]{(BV)}. 
Assume that $(\Gamma, \mu)$ is such a system that  both $(\mathfrak{U}_1)$ and $(\mathfrak{U}_2)$  hold. If  $g((0, 1))\subset (0, 1)$ for every $g\in\Gamma$ and there exist $g_0, g_1\in\Gamma$ with $g_0(0)\in (0, 1)$ and $g_1(1)\in (0, 1)$, then $(\Gamma,\mu)$ is uniquely ergodic.
\end{cor}

The outline of the paper is as follows.
Section~\ref{prelim} contains some notation and definitions.
In Section~\ref{sec: mt} we collect some measure--theoretic results which will be needed for our proofs. In Section~\ref{sec: main} we will use \hyperref[BV]{(BV)} and prove Theorem~\ref{thm:main result} and Corollaries~\ref{clry:main application} and \ref{cor1_16.09.68}. Finally, in Section~\ref{sec: ex} we provide different sufficient conditions for \hyperref[BV]{(BV)} which are easy to check in many examples. Furthermore, in this section, we provide large classes of uniquely ergodic stochastic dynamical systems on $(0,1)$.

\section{Preliminaries}\label{prelim}

By $\mathcal{M}$ we denote the space of all finite signed Borel measures on $[0, 1]$, and by $\mathcal{M}^{+}$ the cone of positive finite Borel measures. Let $\mathcal{P}$ denote the convex subset of all Borel probability measures on $[0, 1]$.
We equip $\mathcal{P}$,  with the {\it Wasserstein distance} $d_W$, i.e.,
$$
d_W(\nu_1, \nu_2)=\sup\left|\int_{[0, 1]} f\d \nu_1 -\int_{[0, 1]} f\d \nu_2\right|\quad\text{for any $\nu_1, \nu_2\in \mathcal P$},
$$
where the supremum is taken over all Lipschitz functions $f: [0, 1]\to\mathbb R$ with the Lipschitz constant less than or equal to $1$ (the so--called $1$--Lipschitz functions).

It is well known that the space $\mathcal{P}$ with the {Wasserstein distance} $d_W$ is a complete metric space and the convergence in the Wasserstein metric is equivalent to the weak convergence of measures.

\medskip

Let $g: [0, 1]\to [0, 1]$, and let $[a, b]\subset [0, 1]$.
By $\bigvee_{a}^b g$ we shall denote the variation of $g$ on $[a, b]$, i.e. 
$$
\bigvee_{a}^b g=\sup s_n(g),
$$
where the supremum is taken over all partitions $a=x_0<x_1<\cdots<x_n=b$ and
$$
s_n(g)=\sum_{i=1}^n |g(x_i)-g(x_{i-1})|.
$$

All properties of the variation of a function used in our arguments can be found in \cite{LM}, Section 6.1 and in \cite{Loj}. 

By $|f|_L$ we denote the Lipschitz constant for a Lipschitz function 
 $f: [0, 1]\to\mathbb R$. 
 
 \begin{remark}\label{r2_09.09.24}
 It is easy to see that if $f$ is a Lipschitz function, then
$$
\bigvee_{a}^b  f g\le |f|_L \bigvee_{a}^b g
$$ 
for any $g: [0, 1]\to [0, 1]$. 
\end{remark}

The following remark may be derived from the definition of variation in the straightforward way.

\begin{remark}\label{r1_09.09.24} For every continuous  function $g: I\to \mathbb R$ ($I$--an interval in $\mathbb R$) such that it admits a countable set $D'$ of points at which it is not differentiable and this set has only finitely many accumulation points and $\sup_{x\in I\setminus D'} |g'(x)|<+\infty$ we have
$|g|_L=\sup_{x\in I\setminus D'} |g'(x)|$. On the other hand,
$\sup_{x\in I\setminus D'} |g'(x)|=\esssup_{x\in I} |g'(x)|$, therefore $|g|_L=\esssup_{x\in I} |g'(x)|$.
\end{remark}



\medskip

\section{Some measure-theoretic auxiliary results}\label{sec: mt}

We start with two results -- which may be of independent interest -- that hold in any Polish space $S$. Analogously as in the case of the interval $[0, 1]$ by $\mathcal{M}(S)$,  $\mathcal{M}^+(S)$ and $\mathcal{P}(S)$ we denote the family of all signed finite Borel measures, all finite Borel measures and all Borel probability measures, respectively.

\begin{lem}\label{lem:order intervals weakly compact}
	Let $S$ be a Polish space and let $\nu_i\in\mathcal{M}(S)$, $i=1,2$, such that $\nu_1\leq\nu_2$. The order interval
	\[
	[\nu_1,\nu_2]\ :=\ \bigl\{\nu\in\mathcal{M}(S)\colon \nu_1\leq\nu\leq\nu_2\bigr\}
	\]
	is compact in every vector space topology on $\mathcal{M}(S)$ that coincides with the weak topology on $\mathcal{M}^+(S)$ and for which $\mathcal{M}^+(S)$  is closed in $\mathcal{M}(S)$.
\end{lem}
\begin{proof}
	One has	$[\nu_1,\nu_2] = \nu_1 + [0,\nu_2-\nu_1]$. Because the topology on $\mathcal{M}(S)$ makes addition continuous, it suffices to show that the order intervals $[0,\nu]$ are compact, for any $\nu\in\mathcal{M}^+(S)$. First, observe that 
	\[
	[0,\nu] = \bigl\{ \hat{\nu}\in\mathcal{M}^+(S)\colon \nu- \hat{\nu}\geq 0\bigr\} = \mathcal{M}^+(S) \cap (\nu-\mathcal{M}^+(S)).
	\]
	$\mathcal{M}^+(S)$ is closed in the chosen vector space topology on $\mathcal{M}(S)$. We conclude that $\nu-\mathcal{M}^+(S)$ is closed and therefore also $[0,\nu]$ in $\mathcal{M}^+(S)$. Since $\nu$ is tight,  for every $\eps>0$ there exists $K_\eps$ compact such that $\nu(S\setminus K_\eps)<\eps$. Then also $\hat{\nu}(S\setminus K_\eps)<\eps$ for every $\hat{\nu}\in[0,\nu]$. So, $[0,\nu]$ is uniformly tight. Prokhorov's Theorem for positive measures (see e.g. \cite{Bog}, Theorem 8.6.2 ) yields that $[0,\nu]$ is compact in the weak topology, hence in the vector space topology on $\mathcal{M}(S)$.
\end{proof}

It is used to obtain the following:

\begin{lem}\label{11_1.08.22}
	Let $S$ be a Polish space an let $\nu_1, \nu_2\in\mathcal M^{+}(S)$ be given. Assume that there exists $\alpha>0$ such that for any $\epsilon>0$ one may find two measures $\nu_1^{\epsilon}, \nu_2^{\epsilon}\in\mathcal P(S)$ with $d_W(\nu_1^{\epsilon}, \nu_2^{\epsilon})<\epsilon$ such that
	$$
	\nu_i\ge\alpha\nu_i^{\epsilon}\quad\text{for $i=1, 2$.}
	$$
	Then the measures $\nu_1, \nu_2$ are not mutually singular.\\
	\end{lem}
	
\begin{proof} Take two sequences of probability measures $(\nu_1^{1/n})_{n\ge 1}$ and $(\nu_2^{1/n})_{n\ge 1}$ satisfying
	$d_W (\nu_1^{1/n}, \nu_2^{1/n})<1/n$ and $\nu_i\ge\alpha
	\nu_i^{1/n}$ for $i=1, 2$. Since the sets
	\begin{equation}\label{e1_1.07.22}
		\{\nu\in\mathcal M^+(S): \alpha\nu\le\nu_i\}\quad\text{for $i=1, 2$,}
	\end{equation}
	are weakly compact in $\mathcal M^+(S)$ according to Lemma \ref{lem:order intervals weakly compact}, we may find an increasing  subsequence of integers $(m_n)_{n\ge 1}$ such that $(\nu_1^{1/m_n})_{n\ge 1}$ and $(\nu_2^{1/m_n})_{n\ge 1}$ are weakly convergent. Moreover, as $d_W (\nu_1^{1/m_n}, \nu_2^{1/m_n})<1/m_n$, they converge to the same measure, say, $\hat\nu$. Obviously $\hat\nu\in\mathcal P(S)$
	and, using again the fact that the sets defined by (\ref{e1_1.07.22}) are weakly compact, hence closed, we finally obtain
	$$
	\nu_i\ge\alpha\hat\nu\qquad\text{for $i=1, 2$,}
	$$
	which completes the proof.
\end{proof}

We continue with technical results that use that the underlying space is $[0,1]$.

\begin{lem}\label{lemma:shift to smaller c}
	Let $\nu\in\mathcal P$ and $\delta\geq 0$ and $0<c\leq 1$ such that $\nu \bigl([0,c)\bigr) >\delta$. Then there exists $\bar{c}\in\supp(\nu)$ such that $0\leq\bar{c}<c$. $\nu\bigl([\bar{c},c)\bigl)>0$ and $\nu \bigl([0,\bar{c}]\bigr)>\delta$.
\end{lem}

\begin{proof}
	Since $\nu\bigl([0,c)\bigr)>0$, $[0,c)\cap\supp(\nu)\neq\emptyset$. Define
	\[
	c^* :=\sup\bigl(\, [0,c)\cap\supp(\nu)\,\bigr).
	\]
	Then $c^*\in\supp(\nu)$ and $0\leq c^*\leq c$. If $c^*=c$, then there exists a stricty increasing sequence $(c_n)$ in $[0,c)\cap\supp(\nu)$ such that $c_n\uparrow c^*$. Then
	\[
	\delta<\nu\bigl( [0,c^*)\bigr) = \nu\bigl( \bigcup_{n=1}^\infty [0,c_n] \bigr) = \lim_{n\to\infty} \nu\bigl( [0,c_n]\bigr).
	\]
	Pick $n_0$ such that $\nu \bigl([0,c_{n_0}]\bigr) >\delta$. Then $c_{n_0+1}\in(c_{n_0},c^*)\cap \supp(\nu)$. So, one gets $\nu\bigl((c_{n_0},c^*)\bigr)>0$. Then $\bar{c}:=c_{n_0}$ has the desired properties.\\
	If $c^*<c$, then $\nu\bigl((c^*,c)\bigr) =0$. Consequently,  $\nu \bigl([0,c^*]) >\delta$. If $\nu(\{c^*\})>0$, then $\bar{c}:=c_*$ satisfies the required properties. If $\nu(\{c^*\})=0$, then $c^*\in\supp(\nu)$ implies that for every $\epsilon>0$, $\nu\bigl( (c^*-\epsilon, c^*+\epsilon)\bigr)>0$. Since $\nu([c^*,c))=0$, for all $\epsilon>0$ sufficiently small, $\nu \bigl( (c^*-\epsilon,c^*) \bigr)>0$. Consequently, there exists a strictly increasing sequence $(c_n)$ in $\supp(\nu)$ with $c_n\uparrow c^*$. Similar reasoning as in the case `$c^*=c$' above, now yields $\bar{c}:=c_{n_0}$ with the required properties.
\end{proof}

Similarly, one has
\begin{lem}\label{lem:increase c}
	Let $\nu\in\mathcal P$, $0<\delta\leq1$ and $0\leq c<1$ such that $\nu \bigl([0,c)\bigr)<\delta$ and $\nu \bigl(\{c\}\bigr)=0$. Then there exists $\tilde{c}\in\supp(\nu)$ such that $c<\tilde{c}\leq 1$, $\nu\bigl((c,\tilde{c}]\bigr)>0$ and $\nu\bigl([0,\tilde{c})\bigr)<\delta$.
\end{lem}
\begin{proof}
	The idea of proof is similar to that of Lemma \ref{lemma:shift to smaller c}. First, since $\nu(\{c\})=0$,
	\[
	\nu\bigl( (c,1]\bigr) = 1 -\nu \bigl([0,c))\bigr) > 1 -\delta\geq 0.
	\]
	Therefore, $(c,1]\cap \supp(\nu) \neq\emptyset$. Define
	\[
	c_* := \inf\bigl(\, (c,1]\cap\supp(\nu)\, \bigr).
	\]
	Then $c_*\in\supp(\nu)$ and $c\leq c_*\leq 1$. If $c_*=c$, then there exists a strictly decreasing sequence $(c_n)$ in $\supp(\nu)$ such that $c_n\downarrow c_*$. Because $\nu (\{c\})=0$, one obtains
	\[
	\delta > \nu\bigl([0,c)\bigr) = \nu \bigl([0,c^*]\bigr) = \nu\bigl( \bigcap_{n=1}^\infty [0,c_n) \bigr) = \lim_{n\to\infty} \nu \bigl( [0,c_n)\bigr).
	\]
	Then one proceeds similar to the proof of Lemma \ref{lemma:shift to smaller c}, also in the case $`c_*>c$'.
\end{proof}

In preparation of the proof of the main result, Theorem \ref{thm:main result}, 
we continue with a crucial property of `ordering of mass' for two mutually singular probability measures. The proof uses elementary techniques, but requires a delicate case-by-case analysis.

\begin{prop}[{\bf Mass ordering principle}]\label{l1_3.08.22}
	Let $\nu_1, \nu_2\in\mathcal P([0,1])$ be mutually singular. Then there exist $a, b\in [0, 1]$, $a<b$ such that 
	$\nu_i([a, b])>0$ for $i=1,2$
	and
	$$
	\text{either}\quad\nu_1([0, b))<\nu_2([0, a])\quad\text{or}\quad
	\nu_2([0, b))<\nu_1([0, a]).
	$$
\end{prop}

\begin{proof}  Since $\nu_1$ and $\nu_2$ are mutually singular, $\nu_1\neq \nu_2$. Hence, there exists $c\in(0,1)$ such that $\nu_1\bigl([0,c]\bigr) \neq \nu_2\bigl([0,c]\bigr)$. Say, $\nu_1\bigl([0,c]\bigr) < \nu_2\bigl([0,c]\bigr)$ (otherwise relabel). Since $\nu_1$ and $\nu_2$ are mutually singular, one must have $\nu_1(\{c\})=0$ or $\nu_1(\{c\})=0$ (or both).\\
Case 1: $\nu_2(\{c\})=0$.\\
Then $0<\nu_2\bigl([0,c]\bigr) = \nu_2\bigl([0,c)]\bigr)$. So, Lemma \ref{lemma:shift to smaller c} yields $0\leq\bar{c}<c$ such that $0<\nu_2\bigl([\bar{c},c)\bigr) = \nu_2\bigl([\bar{c},c]\bigr)$ and
\begin{equation}\label{eq:estimate nu1 nu2}
	\nu_1\bigl( [0,c) \bigr)\ <\  \nu_1\bigl([0,c]\bigr)\ <\ \nu_2\bigl([0,\bar{c}]\bigr).
\end{equation}
If $\nu_1\bigl([\bar{c},c]\bigr) >0$, then we can take $a=\bar{c}$ and $b=c$. If $\nu_1\bigl([\bar{c},c]\bigr) =0$, then in particular $\nu_1(\{c\})=0$. Inequalities \eqref{eq:estimate nu1 nu2} and Lemma \ref{lem:increase c} yields $c<\tilde{c}\leq 1$ such that $\nu_1\bigl((c,\tilde{c}]\bigr)>0$ and $\nu_1\bigl([0,\tilde{c})\bigr) < \nu_2\bigl([0,\bar{c}]\bigr)$. Define $a:= \bar{c}$ and $b:=\tilde{c}$. Then these satisfy the properties in the statement of the proposition.\\
Case 2:  $\nu_1(\{c\})=0$ (and $\nu_2(\{c\})>0$).\\
We have $\nu_1\bigl([0,c)\bigr)<\nu_2\bigl([0,c]\bigr)$ and $\nu_1(\{c\})=0$ in this case. So, Lemma \ref{lem:increase c} yields $c<\tilde{c}\leq 1$ such that $\nu_1\bigl((c,\tilde{c}]\bigr)>0$ and $\nu_1\bigl([0,\tilde{c})\bigr)<\nu_2\bigl([0,c]\bigr)$. Define $a:=c$ and $b:=\tilde{c}$. Then $\nu_2([a,b])\geq \nu_2(\{c\})>0$ and $a,b$ have the desired properties.
\end{proof}

\begin{lem}\label{lem: 0 in supp} Let $\nu$ be a $\mu$-invariant Borel probability measure on $[0,1]$ and assume that $\Gamma=\supp(\mu)$.
If $(\mathfrak{U}_1)$, then  $0\in \supp(\nu)$. Analogously,  $(\mathfrak{U}_2)$ implies that $1\in \supp(\nu)$.
\end{lem}    
\begin{proof} Assume $(\mathfrak{U}_1)$: for every $x\in (0,1]$ there is some function $g\in\Gamma$ such that  $ g(x)<x.$
Since $\Gamma=\supp(\mu)$, we obtain  for every $ x\in(0,1]$ that 
$$
\mu(\{g\in\Gamma\, : \, g(x)<x\})>0.$$ 
Take $z\in \supp(\nu)$ minimal. If $z=0$, the proof of the first statement is completed. If $z>0$, then there exists $g\in \Gamma$ such that $g(z)<z$. Moreover, $$
\mu(\{g\in\Gamma\, : \, g(z)<z\})>0.
$$ 
Further, $\nu([0, z))=0$. On the other hand, by $\mu$-invariance, we see that 
$$\nu([0, z))=\int_{\Gamma} \nu(g^{-1}([0, z)))\, d\mu(g)
\geq 
    \int_{\{g\in\Gamma\, : \, g(z)<z\}}\nu(g^{-1}([0, z)))d\mu(g) >0,
    $$ 
    because $\nu(g^{-1}([0, z)))>0$ due to the fact that 
    $g^{-1}([0, z))$ is an open neighbourhood of $z\in\supp(\nu)$ for every
   $g\in  \{g\in\Gamma\, : \, g(z)<z\}.$ Since this is impossible, $z=0$.

The second statement is proven similarly.
\end{proof}

\section{Proofs of Theorem~\ref{thm:main result} and Corollary~\ref{clry:main application} \& \ref{cor1_16.09.68}}\label{sec: main}
To prove the main theorem let us collect some lemmas.

\begin{lem}\label{l1_10.08.22}
	Let $(\Gamma, \mu)$ be a stochastic dynamical system on $C([0, 1])$. Suppose that there exist two different measures $\nu_1, \nu_2\in\mathcal P([0,1])$, concentrated on $(0,1)$ satisfying
	\[
	\{0,1\}\cap\supp\nu_1\cap\supp\nu_2\neq\emptyset.
	\] 
	Then there exist sequences $(\tilde z_n)_{n\ge 1}$ and $(\hat z_n)_{n\ge 1}$ such that the open intervals $Z_n$ with ends $\tilde z_n$ and $\hat z_n$ are pairwise disjoint, and $\nu_1(Z_n)>0$ and $\nu_2(Z_n)>0$ for $n\ge 1$. 
\end{lem}

\begin{proof} 
 Suppose $0 \in \supp\nu_1\cap\supp\nu_2$. The case when $1$ is in the intersection can be proven similarly.
	   Let $\hat{z}_1\in (0,1]$ be arbitrary. 
Since $0\in\supp(\nu_i)$ and $\nu_i$ are concentrated on $(0,1)$, we obtain $\nu_i((0,\hat{z}_1))>0$ for $i=1,2$. Now take a sequence $y_n\in (0,1)$ converging to $0$. Then there exists an element $y_k$ in this sequence such that $\nu_i((y_k,\hat{z}_1))>0$ for $i=1,2$. Indeed, if not, i.e. $\nu_i((y_n,z_1))=0$ for all $n\in\mathbb{N}$, then by continuity of measures, we obtain $\nu_i((0,\hat{z}_1))=\lim\limits_{n\to \infty} \nu_i(\bigcup_{m=1}^n (y_m,\hat{z}_1))=0$, which violates our assumptions. Set $y_k=:\tilde{z}_1$ and $Z_1=(\tilde{z}_1,\hat{z}_1)$. Now, consider the open neighbourhood $[0,\tilde{z}_1)$ of $0$. Again, $\nu_i((0,\tilde{z}_1))>0$ since $\nu_i$ are concentrated on $(0,1)$. Thus we can repeat the procedure: Set $\tilde{z}_1=:\hat{z}_2$ and find $\tilde{z}_2$ such that $\nu_i((\tilde{z}_2,\hat{z}_2))>0$. Continue like this to find $(\tilde{z}_n)_{n\geq 1}$ and $(\hat{z}_n)_{n\geq 1}$.
\end{proof}

 In the consequences of Lemma \ref{l1_10.08.22} we can achieve:
 
\begin{lem}\label{lem:alternative for Zn}
	Let $(\Gamma, \mu)$ be a stochastic dynamical system on $C([0, 1])$.
	Suppose that there exist two different $\mu$--invariant ergodic Borel probability measures $\nu_1, \nu_2\in\mathcal P$ satisfying
	\[
	\supp\nu_1\cap\supp\nu_2\neq\emptyset.
	\] 
If there is no atomic $\mu$--invariant measure with finite support in the open interval $(0, 1)$,
 then there exist sequences $(\tilde z_n)_{n\ge 1}$ and $(\hat z_n)_{n\ge 1}$ such that the intervals $Z_n$ with ends $\tilde z_n$ and $\hat z_n$ are pairwise disjoint, and $\nu_1(Z_n)>0$ and $\nu_2(Z_n)>0$ for $n\ge 1$. 
\end{lem}	

\begin{proof}
	Choose $z\in\supp\nu_1\cap\supp\nu_2$.  If $z\in\{0, 1\}$, then Lemma \ref{l1_10.08.22} applies and the condition about atomic $\mu$-invariant measures is not needed to draw the conclusion.
	Thus, consider the case $0<z<1$. Since $\nu_1$ and $\nu_2$ are $\mu$--invariant, there exists a measurable set $\widetilde{\Gamma}\subseteq \Gamma$ (depending on $\nu_1,\nu_2$) with $\mu(\widetilde{\Gamma})=1$ such that $g(\supp \nu_1 \cap \supp \nu_2)\subseteq \supp \nu_1 \cap \supp \nu_2$ for all $g\in \widetilde{\Gamma}$. 
	
	Therefore, we have 
	$$
	\varUpsilon:= \bigcup_{n=1}^{\infty} \widetilde{\Gamma}^n(z)\subseteq \supp \nu_1 \cap \supp \nu_2.
	$$ 
	Moreover, 
  we have $\# \varUpsilon=\infty$ by the assumption that there is no
	atomic $\mu$--invariant measure $\nu\in\mathcal P$ with finite support intersecting
	the open interval $(0, 1)$.  Indeed, if $\varUpsilon$ were finite, then by compactness there would be a $\mu$--invariant Borel probability measure on $\varUpsilon$, and its support intersects $(0,1)$ non-trivially and is finite, which contradicts the assumptions.  Obviously 	$\varUpsilon\subset\supp\nu_i$ for $i=1, 2$. 
 Let $(w_n)_{n\ge 1}$ be a sequence of points in $\varUpsilon$. Without loss of generality we may assume that this sequence is monotonic. Let $r_n:=\min \{|w_n-w_{n-1}|/3, |w_n-w_{n+1}|/3\}$ for $n\ge 2$ and $r_1=|w_1-w_2|/3$. Then, obviously, we have $B(w_i, r_i)\cap B(w_j, r_j)=\emptyset$ for $i\neq j$. Again, since $\nu_i(B(w_n, r_n))>0$ for $i=1, 2$ and $n\in\mathbb N$, we may choose $\tilde z_n, \hat z_n\in B(w_n, r_n)$ such that the interval $Z_n$ with ends $\tilde z_n, \hat z_n$ satisfies $\nu_1(Z_n)>0$ and $\nu_2(Z_n)>0$ for $n\in\mathbb N$.
\end{proof}

Finally, let us consider the proof of Theorem \ref{thm:main result}.

\begin{proof}[Proof of Theorem \ref{thm:main result}]
 ({\it i})
	The proof will rely on Lemma \ref{11_1.08.22}. 
	Supposing that, contrary to our claim, there exist two different $\mu$--invariant ergodic measures $\nu_1, \nu_2\in\mathcal P$  satisfying
	$\{0,1\}\cap\supp\nu_1\cap\supp\nu_2\neq\emptyset$, we obtain a contradiction. Namely, 
	we shall prove that $\nu_1, \nu_2$ are not mutually singular, contradicting the well known fact that $\nu_1$ and $\nu_2$ must be mutually singular.	
	Fix an $\epsilon>0$.  From Proposition \ref{l1_3.08.22} it follows that there exist $a, b\in [0, 1]$, $a<b$, such that, say, 
	$$
	\nu_1([0, b))<\nu_2([0, a])
	$$
	and
	$$
	\nu_1([a, b])>0\qquad\text{and}\qquad
	\nu_2([a, b])>0.
	$$
	Set
	$$
	\delta:=\nu_1([a, b])\wedge\nu_2([a, b]).
	$$
 Let us denote for $\omega=(g_1,g_2,\dots)\in \Gamma^\mathbb N$ 
\[
V_n(\omega) := \bigvee_0^1 g_1g_2\dots g_n.
\]	
	Then condition~\hyperref[BV]{(BV)} implies that there exists $K>0$ such that
	$$
	\beta:=\mu^{\otimes \mathbb N}(\omega\in\Gamma^{\mathbb N}: \sup_{n\ge 1} V_n(\omega)<K)>0.
	$$
	Let $\tilde\epsilon=\beta\epsilon{/2}$. Choose $M\in\mathbb N$ such that
	$M>K/\tilde\epsilon$, and let $N\in\mathbb N$ be such that $N\gamma>M$, where
	$$
	\gamma:=
	(\nu_2([0, a])-\nu_1([0, b)))/2>0.
	$$
	By Lemma \ref{l1_10.08.22} we may choose pairwise disjoint intervals $Z_1,\ldots, Z_N$ such that $\nu_1(Z_j)$ $>0$ and $\nu_2(Z_j)>0$ for $j=1,\ldots, N$. Now, by the Individual Birkhoff Ergodic Theorem we choose points
	$x_j, y_j\in Z_j$ such that
	$$
	\lim_{m\to\infty}\frac{1}{m}\sum_{i=1}^m {\bf 1}_{[0, b)} (g_i\cdots g_1(x_j))=\nu_1([0, b))\quad\text{$\mu^{\otimes \mathbb N}$--a.s.}
	$$
	and
	$$
	\lim_{m\to\infty}\frac{1}{m}\sum_{i=1}^m {\bf 1}_{[0, a]} (g_i\cdots g_1(y_j))=\nu_2([0, a])\quad\text{for $j\in\{1,\ldots, N\}$, $\mu^{\otimes \mathbb N}$--a.s. }
	$$
	These equations, in turn, give
	$$
	\lim_{m\to\infty}\frac{1}{m}\sum_{i=1}^m ({\bf 1}_{[0, a]} (g_i\cdots g_1(y_j))-{\bf 1}_{[0, b)} (g_i\cdots g_1(x_j)))=\nu_2([0, a])-\nu_1([0, b))
	$$
	for $j\in\{1,\ldots, N\}$, 
	$\mu^{\otimes \mathbb N}$--a.s.
	Consequently, there exists $n\in\mathbb N$, $n\geq N$, such that
	\begin{equation}\label{e1_11.08.22}
		\frac{1}{n}\sum_{i=1}^n ({\bf 1}_{[0, a]} (g_i\cdots g_1(y_j))-{\bf 1}_{[0, b)} (g_i\cdots g_1(x_j)))>\gamma
	\end{equation}
	for every $j\in\{1,\ldots, N\}$ and every $\omega\in\hat\Omega\subset \Gamma^n$ with $\mu^{\otimes n}(\hat\Omega)>1-\beta/2$.
	Further, we have
	$$
	\begin{aligned}
		\mu^{\otimes n}(\{\omega\in\Gamma^{n}: &  \bigvee_{0}^1g_n\cdots g_1<K\})\ =
		\ \mu^{\otimes n}(\{\omega\in\Gamma^{n}:  \bigvee_{0}^1g_1\cdots g_n<K\})\\
		&\ge\ \mu^{\otimes \mathbb N}(\{\omega\in\Gamma^{\mathbb N}: \sup_{n\ge 1} V_n(\omega)<K\})\ =\ \beta
	\end{aligned}
	$$
	and, thus, we obtain
	$$
	\mu^{\otimes n}(\omega\in\hat\Omega:  \bigvee_{0}^1g_n\cdots g_1<K)\ge\beta/2.
	$$
	Set
	$$
	\Omega:=\{\omega\in\hat\Omega:  \bigvee_{0}^1g_n\cdots g_1<K\}.
	$$
	From (\ref{e1_11.08.22}) it follows that for any $\omega\in\Omega$ and $j\in\{1,\ldots, N\}$ the set 
	$$
	I(j, \omega):=\{1\le i\le n: g_i\cdots g_1(y_j)\le a<b\le g_i\cdots g_1(x_j) \}
	$$
	 has $n_j$ elements with $n_j\geq \gamma\cdot n$. Since $n$ and $N$ had been chosen such that $n\geq N$ and $N\gamma>M\geq 1$, we get $1\leq M < n_j\leq n$. Let $i^{(j)}_1,\dots, i^{(j)}_{n_j}$ be an enumeration of $I(j,\omega)$. Concatenate these into a string of numbers
		\[
		s:=\bigl( i^{(1)}_1,\dots i^{(1)}_{n_1},\ \dots\ , i^{(N)}_1,\dots, i^{(N)}_{n_N}\bigr).
		\]
		Then, $s$ has length $\sum_{j=1}^N n_j \geq N\cdot n\gamma> nM$, while each entry of $s$ has been chosen from $n$ numbers. Therefore, $s$ must contain at least $M$ numbers which are equal. For a fixed $j$, all $i^{(j)}_k$ are distinct. Hence, for any $\omega\in\Omega$ there exists $1\le i< n$ and distinct $j_1(\omega),\dots j_M(\omega)\in\{ 1,\dots,N\}$ such that we have
		$$
		g_i\cdots g_1(x_{j_l(\omega)})\geq b\quad \mbox{and}\quad g_i\cdots g_1(y_{j_l(\omega)})\leq a\quad\text{for all $l=1,\ldots, M$.}
		$$
	Consequently, 
	$$
	|g_n\cdots g_{i+1}([a, b])|<\tilde\epsilon.
	$$
	If not, by the obvious property of the variation, 
	we would obtain 
	$$
	\bigvee_{0}^1g_n\cdots g_1\ge\sum_{l=1}^M
	\bigvee_{x_{j_l(\omega)}}^{y_{j_l(\omega)}} g_n\cdots g_1\ge 
	M |g_n\cdots g_{i+1} ([a, b])|\ge M\tilde\epsilon>K
	\quad\text{for $\omega\in\Omega$,}
	$$
contrary to the definition of the set $\Omega$.
	
	\smallskip
	We are now in a position to define the measures $\nu_1^{\epsilon}$ and $\nu_2^{\epsilon}$. For any $\omega=(g_1, \ldots, g_n)\in\Omega$ we define $T(\omega)=\omega'=(g_n,\ldots, g_l)$, where
	$l$ is the largest integer strictly smaller than $n$ such that 
	$$
	|g_n\cdots g_l ([a, b])|<\tilde\epsilon.
	$$
	
	Set
	$$
	\Omega_i=T(\Omega)\cap \Gamma^i\qquad\text{for $i=1, \ldots, n$}.
	$$
	We easily see that for $i\neq j$ the sets $\Omega_i\times\Gamma^{\mathbb N}$ and $\Omega_j\times\Gamma^{\mathbb N}$ are disjoint and therefore we have
	$$
	\sum_{j=1}^n\mu^{\otimes j}(\Omega_j)=
	\sum_{j=1}^n\mu^{\otimes \mathbb N}(\Omega_j\times\Gamma^{\mathbb N})\ge\mu^{\otimes n}(\Omega)\ge\beta/2.
	$$
	Finally, define 
	$$
	\tilde\nu_i^{\epsilon}(A): =\sum_{j=1}^n\int_{\Omega_j}\frac{\nu_i(g^{-1}(A)\cap [a, b])}{\nu_i([a, b])}\mu^{\otimes i}(\d g)\quad\text{for a Borel set $A$ and $i=1, 2$,}
	$$
	and observe that by the fact that $\nu_1$, $\nu_2$ are $\mu$--invariant measures we obtain
	\begin{equation}\label{e1_12.08.22}
		\begin{aligned}
			\tilde\nu_i^{\epsilon}(A)&=\sum_{j=1}^n\int_{\Omega_i}\frac{\nu_i(g^{-1}(A)\cap [a, b])}{\nu_i([a, b])}\mu^{\otimes j}(\d g)
			\le \sum_{j=1}^n\int_{\Omega_j}\frac{\nu_i(g^{-1}(A))}{\nu_i([a, b])}\mu^{\otimes j}(\d g)\\
			&=\sum_{j=1}^n\int_{ \Omega_j\times\Gamma^{n-j}}\frac{\nu_i(g^{-1}(A))}{\nu_i([a, b])}\mu^{\otimes n}(\d g)
			=\int_{\bigcup_{j=1}^n \Omega_j\times\Gamma^{n-j}}\frac{\nu_i(g^{-1}(A))}{\nu_i([a, b])}\mu^{\otimes n}(\d g)\\
			&\le
			\int_{\Gamma^{n}}\frac{g_*\nu_i(A)}{\nu_i([a, b])}\mu^{\otimes n}(\d g)
			=\frac{\nu_i(A)}{\nu_i([a, b])}\quad\qquad\text{for any Borel set $A$ and $i=1, 2$.}
		\end{aligned}
	\end{equation}
	We see 
	\begin{equation}\label{e1_13.08.22}
		\begin{aligned}
			\tilde\nu_i^{\epsilon}([0, 1]): &=\sum_{j=1}^n\int_{\Omega_j}\frac{\nu_i(g^{-1}([0, 1])\cap [a, b])}{\nu_i([a, b])}\mu^{\otimes j}(\d g)\\
			&=\sum_{j=1}^n\int_{\Omega_j}\frac{\nu_i([a, b])}{\nu_i([a, b])}\mu^{\otimes j}(\d g)= \sum_{j=1}^n\mu^{\otimes j}(\Omega_j)\ge\beta/2
			\quad\text{for $i=1, 2$.}
		\end{aligned}
	\end{equation}
 Note that $\tilde\nu_1^{\epsilon}([0, 1]) = \tilde\nu_2^{\epsilon}([0, 1])$.
	Moreover, for any $1$--Lipschitz function $f$ we have
	\begin{equation}\label{e2_13.08.22}
		\begin{aligned}
			&\left|\int_{[0, 1]} f(x)\tilde\nu_1^{\epsilon}(\d x)-
			\int_{[0, 1]} f(x)\tilde\nu_2^{\epsilon}(\d x)\right|\\
			&\le
			\sum_{i=1}^n\int_{\Omega_i}\left|\int_{[a, b]}f(g(x))\frac{\nu_1(\d x)}{\nu_1([a, b])}-\int_{[a, b]}f(g(x))\frac{\nu_2(\d x)}{\nu_2([a, b])}\right|\mu^{\otimes i}(\d g)\\
			&\le\sum_{i=1}^n\int_{\Omega_i}\left(\int_{[a, b]}\int_{[a, b]} |f(g(x))-f(g(y))|\frac{\nu_1(\d x)}{\nu_1([a, b])}\frac{\nu_2(\d y)}{\nu_2([a, b])}\right)\mu^{\otimes i}(\d g)\\
			&\le\sum_{i=1}^n\int_{\Omega_i}\left(\int_{[a, b]}\int_{[a, b]} |g(x)-g(y)|\frac{\nu_1(\d x)}{\nu_1([a, b])}\frac{\nu_2(\d y)}{\nu_2([a, b])}\right)\mu^{\otimes i}(\d g)
			\le\tilde\epsilon.
		\end{aligned}
	\end{equation}

	Finally, define
	$$
	\nu_i^{\epsilon}(\cdot)=\frac{\tilde\nu_i^{\epsilon}(\cdot)}{
		\tilde\nu_i^{\epsilon}([0, 1])}\qquad\text{for $i=1, 2$}
	$$
	and observe that by (\ref{e1_12.08.22}) and (\ref{e1_13.08.22}) we have
	$$
	\nu_i\ge\alpha\nu_i^{\epsilon}\qquad\text{for $i=1, 2$}
	$$
	with $\alpha=\beta{\delta}/2$ and
	$$
	d_w(\nu_1^{\epsilon}, \nu_2^{\epsilon})<{2}\tilde\epsilon/\beta=\epsilon,
	$$
	by (\ref{e1_13.08.22}) and (\ref{e2_13.08.22}). The application of Lemma \ref{11_1.08.22} completes the proof of ({\it i}).\\

\medskip

({\it ii}) 
Similar as in case ({\it i}), we argue by contradiction. 
	Supposing that contrary to our claim, there exist two distinct $\mu$--invariant ergodic measures $\nu_1, \nu_2\in\mathcal P$ satisfying
	$\supp\nu_1\cap\supp\nu_2\neq\emptyset$, Lemma~\ref{lem:alternative for Zn} instead of Lemma \ref{l1_10.08.22} and the reasoning as in the proof of ({\it i}) shows that $\nu_1, \nu_2$ are not mutually singular, contradicting their ergodicity.
\end{proof}

\begin{proof}[Proof of Corollary~\ref{clry:main application}] Since probability measures are convex combinations of ergodic ones, it suffices to show that there is at most one ergodic $\mu$-invariant probability measure.
Let $\nu_1,\nu_2$ be two $\mu$-invariant ergodic Borel probability measures on $(0,1)$. By Lemma~\ref{lem: 0 in supp}, we obtain that $0\in\supp(\nu_1)\cap\supp(\nu_2)$ or $1 \in\supp(\nu_1)\cap\supp(\nu_2)$. Thus by Theorem~\ref{thm:main result} ({\it i}), we see that $\nu_1=\nu_2$. If moreover, every $\mu$-invariant Borel probability measure gives zero mass to the sets $\{0\}$ and $\{1\}$, then there exists exactly one $\mu$-invariant Borel probability measure on $[0,1]$. Indeed, by compactness, there exists a $\mu$-invariant Borel probability measure $\nu$  on $[0,1]$ and since $\nu(\{0\})=\nu(\{1\})=0$, this measure is concentrated on $(0,1)$. By the above, there can at most be one $\mu$-invariant Borel probability measure on $(0,1)$. This proves unique ergodicity.
\end{proof}

\begin{proof}[Proof of Corollary~\ref{cor1_16.09.68}] To complete the proof it is enough to show that $\nu(\{0\})=\nu(\{1\})=0$ due to  Corollary~\ref{clry:main application}. Assume, contrary to our claim, that $\nu(\{0, 1\})>0$. Since $\nu$ is $\mu$--invariant  and $g((0,1))\subset(0,1)$ for every $g\in\Gamma$ by assumption, we have that $g^{-1}(\{0,1\})$ equals either $\emptyset$, $\{0\}$, $\{1\}$ or $\{0,1\}$ and consequently,
\begin{align*}
\nu(\{0, 1\})\ &=\ \int_{\Gamma} \nu(g^{-1}(\{0, 1\}))\,\mu(\d g)\\
&=\ \int_{\{g\in\Gamma:  g^{-1}(\{0, 1\})= \{0, 1\}\}} \nu(\{0, 1\})\mu(\d g)\\
& \qquad +\ \int_{\bigl\{ g\in\Gamma:  g^{-1}(\{0, 1\})=\{1\}\bigr\} } \nu(\{1\})\mu(\d g)\\
&  \qquad\qquad+\ \int_{ \bigl\{ g\in\Gamma:  g^{-1}(\{0, 1\})=\{0\} \bigr\} } \nu(\{0\})\mu(\d g)\\
& {\leq\ \int_{ \bigl\{g\in\Gamma \colon g(\{0,1\})\subset\{0,1\} \bigr\} } \nu(\{0,1\}) \,\mu(\d g)}\\
&<\ \int_{\Gamma}\nu(\{0, 1\})\d\mu(g)=\nu(\{0, 1\}),
\end{align*}
where the penultimate inequality is due to the fact that the  $\nu(\{0,1\})>0$ and the complement of the set $\bigl\{ g\in\Gamma\colon g(\{0,1\})\subset\{0,1\} \bigr\}$ is non-empty and has positive $\mu$-measure, because it contains $g_0$ and $g_1$ by assumption and $\Gamma=\supp(\mu)$. This leads to a contradiction and the proof is complete.

\end{proof}

\section{Various sufficient conditions and examples}\label{sec: ex}

In this section we shall show that condition \hyperref[BV]{(BV)} can be shown to hold in various examples. First, we construct an example generated by two maps on the unit interval of which one is highly oscillating in regions. It has unbounded variation. However, condition \hyperref[BV]{(BV)} is satisfied due to the particular structure of the other map, that compensates this behaviour. Secondly, we show a construction that allows us to give a large family of examples consisting of $\Gamma$ generated by finitely many piecewise monotone continuous maps. The construction is such that condition \hyperref[BV]{(BV)} is satisfied. By introducing maps that satisfy assumptions ({\it i}) and ({\it ii}) in Corollary \ref{clry:main application}, the existence of a unique $\mu$-invariant measure is then guaranteed by this important corollary. 
	
One family of examples can be immediately provided (see e.g. \cite{AM, CS}):

\subsection{Monotonic maps}  If all $g\in \Gamma$ are monotonic, then
$\bigvee_0^1 g_1\cdots g_n\le 1$ for every $g_1,\ldots, g_n\in\Gamma$ and condition \hyperref[BV]{(BV)} holds. If we additionally assume that both $(\mathfrak{U}_1)$ and $(\mathfrak{U}_2)$  are satisfied, then, according to Corollary \ref{clry:main application}, $(\Gamma, \mu)$ has at most one $\mu$--invariant measure, which is concentrated on $(0, 1)$. Moreover, if there exist $g_0, g_1\in\Gamma$ with $g_0(0)\in (0, 1)$ and $g_1(1)\in (0, 1)$, then $(\Gamma,\mu)$ is uniquely ergodic, by Corollary
\ref{cor1_16.09.68}.

\subsection{$\mu$--injective maps} In  \cite{BOS} the authors introduced the so--called $\mu$--injective maps. Recall that the stochastic dynamical system $(\Gamma, \mu)$ on $C([0, 1])$ is said to be {\it $\mu$--injective} if for every $x\in[0,1]$,
$$
\int_{\Gamma} \# g^{-1}(x)\mu(\d g)\le 1.
$$

Lemma 4.5 in  \cite{BOS} shows that every $\mu$--injective system $(\Gamma, \mu)$ on $C([0, 1])$ satisfies condition \hyperref[BV]{(BV)}. If, additionally, the hypotheses of Corollary \ref{clry:main application} or \ref{cor1_16.09.68} hold, then the system has at most one $\mu$--invariant measure on $(0, 1)$ or is uniquely ergodic.

\subsection{Systems constructed by tableaux}

An important source of examples of random dynamical systems $(\Gamma,\mu)$ with unique ergodicity is provided by the following
\begin{prop}\label{prop: condition gamma-adapted partition}
	Let $\mathcal{I}=\{I_1,\dots, I_m \}$ be a finite covering of $[0,1]$ by closed intervals such that for every $g\in\Gamma$ and for every $I\in\mathcal{I}$:
	
	 {\rm (1)} $g$ is monotone on $I$,  
	 
	 {\rm (2)} there exists $J\in\mathcal{I}$ such that $g(I)\subset J$.\newline Then for any probability measure $\mu$ on $C([0,1])$ with $\supp(\mu)=\Gamma$, condition \hyperref[BV]{(BV)} is satisfied.
\end{prop}

\begin{proof} 
 Let $\omega = (g_1,g_2,\dots)\in\Gamma^\mathbb N$, and let $\mu\in\mathcal P$ be given. One has for any $g\in\Gamma$, according to the monotonicity of $g$ on any $I\in\mathcal{I}$:
\[
	\bigvee_0^1 g \leq \sum_{I\in\mathcal{I}} \bigvee_I g \leq \sum_{I\in\mathcal{I}} |g(b_I) - g(a_I)|\ \leq\ |\mathcal{I}|,
\]
where $I=[a_I,b_I]$ and $|\mathcal{I}|$ indicates the number of intervals in $\mathcal{I}$. So all $g\in \Gamma$ are of bounded variation.
The assumptions satisfied by any $g\in\Gamma$ imply further the existence of a map $\phi_g:\mathcal{I}\to\mathcal{I}$ such that $g(I)\subset\phi_g(I)$. Thus, for given $\omega$ and any $k\in \mathbb N$, $k\geq 2$, define $\phi_\omega^k:=\phi_{g_2}\circ\dots\circ\phi_{g_k}$. Then, because each $g$ is monotonic on the intervals from $\mathcal{I}$, for any $n\in\mathbb N$, $n\geq 2$,
	\begin{align*}
		V_n(\omega)\ &=\ \bigvee_0^1 g_1g_2\cdots g_n \ \leq \sum_{I\in\mathcal{I}} \bigvee_I g_1g_2\cdots g_n
		\ \leq \ \sum_{I\in\mathcal{I}} \bigvee_{\phi_{g_n}(I)} g_1g_2\cdots g_{n-1}\\
		&\leq \ \sum_{I\in\mathcal{I}} \bigvee_{\phi_\omega^n(I)} g_1\ \leq \ |\mathcal{I}|\cdot \max_{I\in\mathcal{I}} \bigvee_I g_1 \ \leq\ |\mathcal{I}|\,\bigvee_0^1 g_1\ \leq\ |\mathcal{I}|^2.
	\end{align*}
	Thus, $\sup_n V_n(\omega)<\infty$ for any $\omega$.
\end{proof}

\begin{remark}\label{rem: interval crossing cond}
The assumptions of Proposition~\ref{prop: condition gamma-adapted partition} are in particular fulfilled under the following condition, which is often easy to verify and thus provides a large class of examples which fulfil \hyperref[BV]{(BV)}:\\
  Assume there exists an essential partition   $\mathcal{I}=\{I_1,\dots, I_m \}$ of $[0,1]$ into intervals such that 
  \begin{enumerate}
		\item[(1)] every $g\in \Gamma$ is monotone on each $I\in\mathcal{I}$.	
		\item[(2)] crossing of the graph of a $g\in \Gamma$ from $I_j\times I_k$ to $I_{j'}\times I_{k+1}$ or $I_{j'}\times I_{k-1}$ (i.e. changing partition interval in vertical direction) can only happen at a corner point of $I_j\times I_k$.
	\end{enumerate} 
 Then, for any probability measure $\mu$ on $C([0,1])$ with $\supp(\mu)=\Gamma$, the assumptions of Proposotion~\ref{prop: condition gamma-adapted partition} are satisfied.
\end{remark}

By Corollary~\ref{clry:main application}, Proposition~\ref{prop: condition gamma-adapted partition} above implies the following statement about the quantity of $\mu$-invariant Borel probability measures on $(0,1)$:

\begin{cor}\label{cor: max one prob}
Let the assumptions of Proposition \ref{prop: condition gamma-adapted partition} be satisfied. If, moreover,  condition $(\mathfrak{U}_1)$ or $(\mathfrak{U}_2)$ holds then the stochastic dynamical system $(\Gamma,\mu)$ has at most one $\mu$-invariant Borel probability measure on $(0,1)$.
\end{cor}

Let us now provide a concrete example that falls in the above class. 
Let us consider the discrete system $\Gamma=\{g_1, g_2, g_3\}$ defined by the drawing below. 
\begin{center}
\begin{tikzpicture}
\draw (0,0) rectangle (4,4);
\draw[densely dotted] (0,0) -- (4,4);
\draw[loosely dotted] (0,3) -- (4,3);
\draw[loosely dotted] (0,2) -- (4,2);
\draw[loosely dotted] (0,1) -- (4,1);
\draw[loosely dotted] (1,0) -- (1,4);
\draw[loosely dotted] (2,0) -- (2,4);
\draw[loosely dotted] (3,0) -- (3,4);
\draw[blue]  (0,0) to[out=30,in=190] (1,0.7) ;
\draw[blue]  (1,0.7) to[out=10,in=210] (2,1) node at (1.5,0.5) {$g_1$};
\draw[blue]  (2,1) to[out=-10,in=170] (3,0.4);
\draw[blue]  (3,0.4) to[out=-20,in=175] (4,0);
\draw[red]  (0,3) to[out=-55,in=160] (1,2) node at (0.5,2) {$g_2$};
\draw[red]  (1,2) to[out=10,in=240] (2,3);
\draw[red]  (2,3) to[out=35,in=190] (4, 4);
\draw[green]  (0,4) to[out=-45,in=160] (0.5,3.5);
\draw[green]  (0.5,3.5) to[out=-25,in=140] (1,3);
\draw[green]  (1,3) to[out=-45,in=160] (2,2);
\draw[green]  (2,2) to[out=-45,in=170] (3,1);
\draw[green]  (3,1) to[out=15,in=200] (4,1.5) node at (3,1.4) {$g_3$};

\draw (0,-0.05)--(0,0) node[anchor=north]{$0$};
\draw (-0.05,4)--(0.05,4) node[anchor=east]{$1$};
\draw (4,-0.05)--(4,0) node[anchor=north]{$1$};
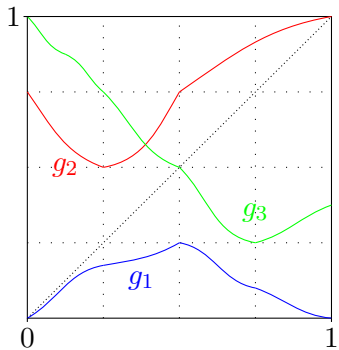
\node at (2,-0.8) {\textbf{A family of functions satisfying Proposition~\ref{prop: condition gamma-adapted partition}}};
\end{tikzpicture}
\end{center}
 Let $\mu$ be any probability measure on $\Gamma$ with $\mu(\{g_i\})\neq 0$ for all $i=1,2,3$. The dotted lines in the drawing indicate a partition like in the condition of Remark~\ref{rem: interval crossing cond}. Thus the system satisfies \hyperref[BV]{(BV)}. Moreover, conditions $(\mathfrak{U}_1)$ 
 and $(\mathfrak{U}_2)$
 are satisfied ($g_1$ is below the diagonal and $g_2$ is above the diagonal). Finally, we have
 $g_i((0, 1))\subset (0, 1)$ and $g_2(0)\in (0, 1)$ and $g_3(1)\in (0, 1)$. Thus, by Corollary \ref{cor1_16.09.68}, the system $(\Gamma, \mu)$ is uniquely ergodic.

\subsection{\bf  Compensation of unbounded variation.}
Consider a sequence of closed intervals $I_n:=[a_n,b_n]$ with non-empty interior, contained in $[0,1]$, disjoint and such that $b_{n+1}<a_n$, while $a_n\downarrow 0$. Put $a_0:=1$, so $b_1<1$. Let $J_n:=[b_{n}, a_{n-1}]$ be the closure of the intermediate intervals. Thus, the intervals $I_n$, $J_m$ are essentially pairwise disjoint and
\[
[0,1] = \{0\}\ \cup\ \bigcup_{n=1}^\infty ( \ I_n\ \cup\ J_n).
\]
	
We shall define continuous maps $f_i:[0,1]\to[0,1]$, $i=1,2$, piecewise linear, such that $f_1(x)<x$ for all $0<x\leq 1$ and $f_2(x)>x$ for all $0\leq x < 1$. Moreover, the construction will be so that $f_1(I_n)= I_{n+1}$ and $f_1(J_n)= J_{n+1}$, while $\bigvee_{a_n}^{b_n} f_1 \uparrow \infty$ as $n\to \infty$. $f_2$ will be monotone on $[0,1]$.

Let $(m_n)$ be a strictly increasing sequence satisfying
\begin{equation}\label{e1_2.09.24}
(2m_n -1) |I_{n+1}|\to+\infty\qquad\text{as $n\to+\infty$.}
\end{equation}
Put $\Delta_n:=|I_n|/(2m_n-1)$  for $n\geq 1$. We define $f_1$ on $I_n$ as the piecewise linear continuous function that is linear on each subinterval $[a_n+ k\Delta_n, a_n+(k+1)\Delta_n]$, ($k=0,1,\dots,2m_n-1$), and
\begin{equation}
	f_1(a_n +k\Delta_n) := \begin{cases}
		a_{n+1} & \quad \mbox{for}\ k\ \mbox{even},\  0\leq k\leq 2m_n-2
  ,\\
		b_{n+1} & \quad \mbox{for}\ k\ \mbox{odd},\   1\leq k\leq 2m_n-1,

	\end{cases}
\end{equation}
  for $n\geq 1$ and $f_1(a_0):=a_1$.
Thus, $f_1$ on $I_n$ has $m_n$ maximal values $b_{n+1}$ and $m_n$ minimal values $a_{n+1}$ and `oscillates linearly' between these. Note that $f_1(a_n)=a_{n+1}$ and $f_1(b_n) = b_{n+1}$. 
\begin{center}
\begin{tikzpicture}[scale=0.50]
\draw (0,0) rectangle (12,12);
\draw (0,0) -- (12,12);
\draw (12,9) -- (10,8); 
\draw (9,7) -- (10,8); 
\draw (9,7) -- (7.66666, 5); 
\draw (7.66666,5) -- (7.333333,6); 
\draw (7.333333,6) -- (7,5); 
\draw (7,5) -- (6, 4); 
\draw (6,4) -- (5.8, 3); 
\draw (5.8,3) -- (5.6, 4); 
\draw (5.6,4) -- (5.4, 3); 
\draw (5.4, 3) -- (5.25, 4); 
\draw (5.25, 4) -- (5, 3); 
\draw (5,3) -- (4, 2.5);
\draw (4, 2.5) -- (3.8571428,2);
\draw (3.8571428, 2) -- (3.7142857, 2.5);
\draw (3.7142857, 2.5) -- (3.5714285, 2);
\draw (3.5714285, 2) -- (3.428571, 2.5);
\draw (3.428571, 2.5) -- (3.285714, 2);
\draw (3.285714, 2) -- (3.142857, 2.5);
\draw (3.142857, 2.5) -- (3,2);
\draw (3,2) -- (2.5, 1.75);
\draw (2.5, 1.75) -- (2.444444, 1.5);
\draw (2.444444,1.5) -- (2.3888888, 1.75);
\draw (2.3888888,1.75) -- (2.333333, 1.5);
\draw (2.3333333,1.5) -- (2.277777, 1.75);
\draw (2.277777,1.75) -- (2.222222, 1.5);
\draw (2.222222, 1.5) -- (2.1666666,1.75);
\draw (2.1666666,1.75) -- (2.11111, 1.5);
\draw (2.11111, 1.5) -- (2.05555, 1.75);
\draw (2.05555, 1.75) -- (2,1.5);
\draw (2,1.5)-- (1.75,1.25);
\draw (1.75, 1.25) -- (1.7272727272,1);
\draw (1.7272727272,1) -- (1.7045454545,1.25);
\draw (1.7045454545, 1.25) -- (1.6818181818, 1);
\draw(1.6818181818, 1) -- (1.659090909,1.25);
\draw (1.659090909,1.25) -- (1.636363636,1);
\draw (1.636363636,1) -- (1.613636363636,1.25);
\draw (1.613636363636, 1.25) -- (1.5909090909,1);
\draw (1.5909090909, 1) -- (1.5681818181, 1.25);
\draw (1.5681818181, 1.25) -- (1.54545454545,1);
\draw (1.54545454545, 1) -- (1.52272727,1.25);
\draw (1.52272727,1.25) -- (1.5,1);
\draw[densely dotted] (1.5,1) -- (1.2,0.8); 
\draw[loosely dotted] (0,10) -- (12,10);
\draw[loosely dotted] (0,9) -- (12,9);
\draw[loosely dotted] (0,8) -- (12,8);
\draw[loosely dotted] (0,7) -- (12,7);
\draw[loosely dotted] (0,6) -- (12,6);
\draw[loosely dotted] (0,5) -- (12,5);
\draw[loosely dotted] (0,4) -- (12,4);
\draw[loosely dotted] (0,3) -- (12,3);
\draw[loosely dotted] (0,2.5) -- (12,2.5);
\draw[loosely dotted] (0,2) -- (12,2);
\draw[loosely dotted] (0,1.75) -- (12,1.75);
\draw[loosely dotted] (0,1.5) -- (12,1.5);
\draw[loosely dotted] (0,1.25) -- (12,1.25);
\draw[loosely dotted] (0,1) -- (12,1);
\draw (-0.05,12)--(0.05,12) node[anchor=east]{$a_0=1$};
\draw (-0.05,10)--(0.05,10) node[anchor=east]{$b_1$};
\draw (-0.05,9)--(0.05,9) node[anchor=east]{$a_1$};
\draw (-0.05,8)--(0.05,8) node[anchor=east]{$b_2$};
\draw (-0.05,7)--(0.05,7) node[anchor=east]{$a_2$};
\draw (-0.05,6)--(0.05,6) node[anchor=east]{$b_3$};
\draw (-0.05,5)--(0.05,5) node[anchor=east]{$a_3$};
\draw (-0.05,4)--(0.05,4) node[anchor=east]{$b_4$};
\draw (-0.05,3)--(0.05,3) node[anchor=east]{$a_4$};
\draw (-0.05,2.5)--(0.05,2.5) node[anchor=east]{$b_5$};
\draw (-0.05,2)--(0.05,2) node[anchor=east]{$a_5$};
\draw (-0.05,1.5)--(0.05,1.5) node[anchor=east]{$a_6$};
\draw (-0.05,1)--(0.05,1) node[anchor=east]{$a_7$};
\draw (-0.05,0)--(0.05,0) node[anchor=east]{$0$};
\draw (12,0)--(12.05,0) node[anchor=west]{$1$};
\node at (6,11.5) {\textbf{$f_1$ for choosing $m_n=n$.}};
\end{tikzpicture}
\end{center}


Moreover, 
\begin{equation}
\esssup_{x\in I_n}  |f_1'(x)| = \frac{b_{n+1}-a_{n+1}}{\Delta_n} = \frac{|I_{n+1}|}{|I_n|} (2m_n-1).
\end{equation}
Further, $f_1(0):=0$ and, finally, on $J_n$ $f_1$ is the linear increasing function such that $f_1(b_n)=b_{n+1}$ and $f_1(a_{n-1})=a_n$. In particular, $f_1(1) = f_1(a_0) = a_1<1$. Then we have
\begin{equation}
\esssup_{x\in J_n}  |f_1'(x)| = \frac{|J_{n+1}|}{|J_n|}.
\end{equation}
Note that by construction $f_1(x)<x$ for all $0<x\leq 1$.
Since
\[
\bigvee_0^1 f_1 \ge \bigvee_{a_n}^{b_n}f_1=(2m_n-1) |I_{n+ 1}|\to\infty\qquad\text{as $n\to+\infty$},
\]
by (\ref{e1_2.09.24}), 
we obtain $\bigvee_0^1 f_1=+\infty$.

We now define $f_2$ on the intervals $I_n$ and $J_n$ as piecewise linear with slope $0<S^I_n<1$ and $0<S^J_n<1$ respectively -- to be chosen later -- and $f_2(1)=1$. Thus, $f_2$ is monotone increasing on $(0,1]$ and $f_2(x)>x$ for $0<x<1$. Hence, $\lim_{x\downarrow 0} f_2(x)$ exists. Extend $f_2$ continuously to $0$. Then necessarily, $f_2(0)>0$.

Since $f_2$ has only countably many points where it is not differentiable and $0$ is the only accumulation point, Remark \ref{r1_09.09.24} gives
\begin{equation}\label{eq:f2 is non-expansive}
	|f_2|_L = \esssup_{x\in [0, 1]}|f_2'(x)|= \max\bigl( \sup_{n\geq 1} S^I_n, \sup_{n\geq 1} S^J_n\bigl) \leq 1,
\end{equation} where $|f_2|_L$ denotes the Lipschitz constant of $f_2$. In particular, we see that $\bigvee_0^1 f_2\leq 1$.

For any $m\in\mathbb N$ one has
\begin{equation}\label{eq:Lipschits constant f2 before f1-s}
	\bigl| f_2 f_1^m\bigr|_{L}= \esssup_{x\in [0, 1]}\, \bigl| f_2'\bigl(f_1^m(x)\bigr) \cdot f_1'\bigl(f_1^{m-1}(x)\bigr)\cdots f_1'(x) \bigr|,
\end{equation}
by the fact that $f_2 f_1^m$ satisfies the hypotheses of Remark \ref{r1_09.09.24}. 

We shall estimate the essential supremum in \eqref{eq:Lipschits constant f2 before f1-s} by considering the essential suprema on the intervals $I_n$ and $J_n$ separately.
Note that $f_1(I_n)=I_{n+1}$ for $n\geq 1$.Thus,
\begin{align*}
	E^{I,m}_n \ &:=\ \esssup_{x\in I_n}\, \bigl| f_2'\bigl(f_1^m(x)\bigr) \cdot f_1'\bigl(f_1^{m-1}(x)\bigr)\cdots f_1'(x) \bigr| \\
	&\leq\ \esssup_{x\in I_{n+m}}\, \bigl| f'_2(x)\bigr| \cdot \prod_{k=0}^{m-1} \esssup_{y\in I_{n+k}}\, \bigl| f'_1(y)\bigr|\\
	&=\ S^I_{n+m} \cdot \prod_{k=0}^{m-1} \frac{|I_{n+k+1}|}{|I_{n+k}|}(2m_{n+k}-1)\ =\ S^I_{n+m} \frac{|I_{n+m}|}{|I_{n}|} \prod_{k=0}^{m-1}(2m_{n+k}-1).
\end{align*}
Similarly,
\begin{align*}
	E^{J,m}_n \ &:=\ \esssup_{x\in J_n}\, \bigl| f_2'\bigl(f_1^m(x)\bigr) \cdot f_1'\bigl(f_1^{m-1}(x)\bigr)\cdots f_1'(x) \bigr| \\
	&\leq\ \esssup_{x\in J_{n+m}}\, \bigl| f'_2(x)\bigr| \cdot \prod_{k=0}^{m-1} \sup_{y\in J_{n+k}}\, \bigl| f'_1(y)\bigr|
	\ =\ S^J_{n+m} \cdot \prod_{k=0}^{m-1} \frac{|J_{n+k+1}|}{|J_{n+k}|}\\
	&=\ S^J_{n+m} \frac{|J_{n+m}|}{|J_{n}|}.
\end{align*}

We may choose the intervals such that $|J_{n+1}|/|J_n|\leq \theta<1$ for all $n$. One has $|I_{n+1}|/|I_n|<1$ by construction. So, if we take 
\begin{equation}
	S^I_\ell\ \leq\ \prod_{k=1}^{\ell-1} (2m_k-1)^{-1},
\end{equation}
then $E^{I,m}_n\leq 1$ for all $n,m\in\mathbb N$. Moreover, since we are taking $S^J_\ell<1$ for all $\ell$ to assure that $f_2(x)>x$, we find that $E^{J,m}_n<\theta^m<1$ for all $n,m\in\mathbb N$ too. Thus we obtain:
\begin{lem}\label{lem:properties f1 f2}
	With the choices for $[a_n,b_n]$ and $S_n^I$ and $S_n^J$ as described above the constructed functions $f_1$ and $f_2$ are continuous on $[0,1]$ and such that: 
	\begin{enumerate}
		\item[({\it i})] $f_1$ has unbounded variation on $[0,1]$.
		\item[({\it ii})] $f_1(x)<x$ on $(0,1]$.
		\item[({\it iii})] $f_2(x)>x$ on $[0,1)$.
		\item[({\it iv})] For any $m\in \mathbb N$, $f_2f_1^m$ is Lipschitz on $[0,1]$ and $|f_2f_1^m|_L\leq 1$.
	\end{enumerate} 
\end{lem}

Property ({\it iv}), combined with the non-expansiveness of $f_2$ (see \eqref{eq:f2 is non-expansive}), yields that 
\begin{cor}
	For any infinite sequence $\omega=(f_2, f_{i_1},f_{i_2},\dots)$ of $f_i$ ($i=1,2$) that starts with $f_2$, one has for all $n\in\mathbb N$,
	\[
		V_n(\omega) := \bigvee_0^1 f_2 f_{i_1}\cdots f_{i_{n-1}} \leq 1,
	\] for $i_j\in\{1,2\}$.
\end{cor}
\begin{proof}
	We prove the above corollary by induction on $n$. For $n=1$ the result is fine (Equation \eqref{eq:f2 is non-expansive}). For $n=2$, if $i_1=2$ then $V_1 \bigvee_0^1 f_2 f_2\leq 1$, by the fact that $f_2$, and also $f_2 f_2$, is a monotonic function. If $i_1=1$, then the result is simply
	Lemma \ref{lem:properties f1 f2} ({\it iv}). Suppose the estimate has been proven for all $n$ up to $n=N$   for all $\omega\in\{f_2\}\times \{f_1,f_2\}^{\mathbb{N}}$. Let $\omega=(f_2,f_{i_1},\ldots )\in\{f_2\}\times \{f_1,f_2\}^{\mathbb{N}}$ be arbitrary given. If the $i_1,\dots ,
 i_{N+1}
 $ are either all 1 or all are 2, then the result follows from either Lemma \ref{lem:properties f1 f2} ({\it iv}) or non-expansiveness of $f_2$, respectively. So, let $n_0$ be the smallest integer such that $i_{n_0}=1$. Let $m$ be the largest non-negative integer such that $i_{n}=1$ for all $n_0\leq n\leq n_0+m$.  Let $T: \{f_1,f_2\}^{\mathbb{N}}\longrightarrow \{f_1,f_2\}^{\mathbb{N}}$ denote the shift $T(f_{i_1},f_{i_2},\ldots):= (f_{i_2},\ldots)$. If $n_0=1$, then  $T^{m+2}\omega$ starts with $f_2$. Thus, we obtain (see Remark \ref{r2_09.09.24}) that
	\[
		V_{N+1}(\omega) \leq |f_2f_1^{m+1}|_L\cdot V_{N-m-1}(T^{m+2}\omega) \leq 1,
	\]
	according to Lemma \ref{lem:properties f1 f2} ({\it iv}) and the induction hypothesis. If $n_0>1$, then we analogously obtain
	\[
		V_{N+1}(\omega) \leq |f_2^{n_0-1}|_L\cdot V_{N-n_0+2}(T^{n_0-1}\omega) \leq |f_2|_L^{n_0-1} \cdot V_{N-n_0+2}(T^{n_0-1}\omega).
	\] 
	Now, $\omega':=T^{n_0-1}\omega$ is an infinite sequence that starts with $f_2$. Hence, $V_{N-n_0+2}(\omega')\leq 1$ by the induction hypothesis. The estimate follows.
\end{proof}

Since the system $(\Gamma, \mu)$ for $\Gamma=\{f_1, f_2\}$ and $\mu(f_1)=p$, $\mu(f_2)=1-p$ for some $p\in (0, 1)$ satisfies condition (BV),  $0, 1\in\supp\nu$ for every $\mu$--invariant measure $\nu$, $f_i((0, 1))\subset (0, 1)$,  $f_1(1)\in (0, 1)$ and $f_2 (0)\in (0, 1)$, Corollary \ref{cor1_16.09.68} implies the following:
\begin{prop} The system $(\Gamma, \mu)$
for arbitrary $p\in (0, 1)$ admits a unique $\mu$--invariant ergodic measure on $(0, 1)$.
\end{prop}


\begin{thebibliography}{10}




	
	\bibitem[{AM}]{AM} L. Alsed\`{a} and M. Misiurewicz, \textit{Random interval homeomorphisms}, Proceedings of New Trends in Dynamical Systems, Salou, 2012, Publ. Mat., (2014), 15--36.
	
	
	

	
	
\bibitem[{BS}] {BS} K. Bara\'nski, A. \"Spiewak, {\it Singular stationary measures for random piecewise affine interval homeomorphisms}, J. Dynam. Differential Equations 33 (2021), no. 1, 345--393. 
	
\bibitem[{BBS}]{BBS}S. Brofferio, D. Buraczewski, and T. Szarek, {''On uniqueness of invariant measures for random walks on $HOMEO^+(\mathbb{R})$''}, Ergodic Theory and Dynamical Systems (2021), 1--32. 

\bibitem[Bog]{Bog} V. I. Bogachev, \textit{Measure Theory},  volume 1,2,  Springer-Verlag, Berlin, 2007.


 
\bibitem[BOS]{BOS} S. Brofferio, H. Oppelmayer, and T. Szarek, {\it Unique ergodicity for random noninvertible maps on an interval}, arXiv: 2401.12361v2



	
	\bibitem[{C}]{C} D. Czapla, \textit{On the Existence and Uniqueness of Stationary Distributions for Some Class of Piecewise Deterministic Markov Processes With State-Dependent Jump Intensity}, preprint,  arXiv:2303.11576v4 [math.PR] .
	
\bibitem[{CS}]{CS} K. Czudek and T. Szarek, {''Ergodicity and central limit theorem for random interval homeomorphisms''}, Israel J. Math. 239 (2020), no. 1, 75--98.

\bibitem[{DPZ}]{DPZ} G. Da Prato, J. Zabczyk, \textit{Ergodicity for Infinite Dimensional Systems}, Cambridge Monographs on Particle Physics, Nuclear Physics and Cosmology, 1996, Cambridge University Press, ISBN: 9781107362499.
 
\bibitem[{DKN}]{DKN} B. Deroin, V. Kleptsyn, and A. Navas, {''Sur la dynamique unidimensionnelle en r\'egularit\'e interm\'ediaire''}, Acta Math. {199} no. 2, (2007), 199--262. 

\bibitem[{DKNP}]{DKNP} B. Deroin, V. Kleptsyn, A. Navas, and K. Parwani, {''Symmetric random walks on  $Homeo^+(\mathbb{R})$''}, Ann. Probab. 41(3B) (2013), 2066--2089.

	
	
	\bibitem[{F}]{F} H. Furstenberg, \textit{Random walks and discrete subgroups of Lie groups},  Advances in Probability and Related Topics, Vol. 1, pp. 1--63, Dekker, New York, 1971.
	
	\bibitem[{GH}]{GH} M. Gharaei and A. J. Homburg, \textit{Random interval diffeomorphisms}, Discrete Contin. Dyn. Syst. Ser. S 10 (2017), no. 2, 241--272. 

\bibitem[HK]{HK} F. Hahn, Y. Katznelson, \textit{On the entropy of uniquely ergodic transformations,}
Trans. Amer. Math. Soc., Vol. 126 (1967), pp. 335-360.

\bibitem[{HHOS}]{HHOS} S.C. Hille, K. Horbacz, H. Oppelmayer, T. Szarek, \textit{Proximality, stability, and central limit theorem for random maps on an interval}, 	arXiv:2408.07398 [math.DS].
 
	\bibitem[{HKRVZ}]{HKRVZ} A.J. Homburg, C. Kalle, M. Ruziboev, E. Verbitskiy, and B. Zeegers,
	\textit{Critical intermittency in random interval maps},
	Comm. Math. Phys. 394 (2022), no. 1, 1--37. 


	
		\bibitem[Ka]{Ka} S. Kakutani, \textit{Ergodic theory of shift transformations}, Proceedings of the Fifth Berkeley
Symposium on Mathematical Statistics and Probability, Berkeley and Los Angeles, University
of California Press, 1967, Vol. 2, pp. 405--414.

	
	\bibitem[{Kl}]{Kl} B. Kloeckner, \textit{Optimal transportation and stationary measures for Iterated Function Systems},
	Math. Proc. of the Cambridge Philos. Soc. 173 (2022), no. 1. 163--187.

 
On unique ergodicity
\bibitem[{Kr}]{Kr} W. Krieger, \textit{On unique ergodicity}, 
Editors: Lucien M. Le Cam, Jerzy Neyman, Elizabeth L. Scott in 
Berkeley Symp. on Math. Statist. and Prob. vol. 6(2) (1972) 327--346.
 
	
\bibitem[{LM}]{LM} A. Lasota and M. C. Mackey, \textit{Chaos, Fractals, and Noise: Stochastic Aspects of Dynamics}, Applied Mathematical Sciences, vol. 97 (New York: Springer-Verlag), 1994. 
	
	
\bibitem[{L}]{L} M. Lin, \textit{Conservative Markov processes on a topological space.} Israel Journal of Mathematics 8(2) (1970), 165--186. 

\bibitem[{Loj}]{Loj} S. \L ojasiewicz, \textit{An introduction to the theory of real functions},  A Wiley-Interscience Publication. John Wiley \& Sons, Ltd., Chichester, 1988.
	
	
	\bibitem[{M}]{M} D. Malicet, \textit{Random Walks on $Homeo(S^1)$}, Comm. Math. Phys. 356 (2017), 1083--1116. 
	
	
	
	\bibitem[{N}]{N} A. Navas, \textit{Groups of circle diffeomorphisms}, Chicago Lectures in Mathematics, University of Chicago Press, Chicago, IL, 2011. 
	

\bibitem[O]{O} J. C. Oxtoby, \textit{Ergodic sets,} Bull. Amer. Math. Soc., Vol. 58 (1952), pp. 116--136.
\end{thebibliography}
\end{document}